\documentclass[11pt]{amsart}
\usepackage[usenames]{color}

\usepackage{latexsym}
\usepackage{amsfonts}
\usepackage{amsthm}
\usepackage{graphicx}
\usepackage{amsmath,hyperref}
\usepackage{amssymb}
\usepackage{enumitem}
\usepackage{amsmath, amsfonts, amsthm, float, amscd, latexsym, amssymb, xcolor, tikz, dsfont, cite, inputenc, faktor, mathrsfs, mathtools, hyperref, mathbbol, caption}
\usepackage{bbm}
\usepackage[all,cmtip]{xy}
\input xypic 
\usepackage{tikz}
\usepackage{todonotes}

\newcommand{\Z}{{\mathbb Z}}
\def\xa{(\underline{X},\underline{A})}

 \newtheorem{thm}{Theorem}[section]
 \newtheorem{defn}[thm]{Definition}
 \newtheorem{cor}[thm]{Corollary}
 \newtheorem{lem}[thm]{Lemma}
 
 \theoremstyle{definition}
 \theoremstyle{remark}
 
 \newtheorem{ex}{Example}
 \numberwithin{equation}{section}
 
 \numberwithin{equation}{section}

\begin{document}

\title[On the Cohomology Ring of Real Moment-Angle Complexes]
{On the Cohomology Ring of Real Moment-Angle Complexes }

\author[E.Vidaurre]{Elizabeth Vidaurre}
\address{Molloy College}%
\email{evidaurre@molloy.edu}%
\date{\today}

\begin{abstract}
In this article, we study the cohomology ring of real moment-angle complexes over a simplicial complex $K$. Combinatorial generators for the cohomology can be given in terms of $K$. For $K$ the boundary of an $n$-gon, we give a full description of the multiplicative structure of the cohomology ring in terms of the combinatorial generators. As a consequence, it is evident that these generators do not form a symplectic basis, unlike the case for moment-angle complexes.

\end{abstract}

\maketitle

\section{Introduction}
Fixing a pair of topological spaces $(X,A)$, polyhedral product spaces $Z_K({X},{A})$ give a family of spaces where $K$ is a simplicial complex  (see Definition \ref{func}). Examples include moment-angle complexes, complements of complex
coordinate subspace arrangements, and intersections of quadrics among others. In certain cases, polyhedral products provide geometric realizations of right-angled Artin groups and the Stanley-Reisner ring (see Definition \ref{SRring}).


The real moment-angle complex, $Z_K(D^1,S^0)$, and its complex analog (arising from the pair of spaces, the unit disc $D^2$ and the circle $S^1$) feature in toric topology, as they have been key in showing applications in combinatorics and algebraic geometry, among others \cite{MR3363157}. The cohomology ring of the moment-angle complex is shown to be isomorphic to the Tor-algebra $Tor_{\Z[v_1,\ldots, v_m]}(\Z[K],\Z)$ in \cite{MR2117435}, where $\Z[K]$ is the Stanley-Reisner (or face ring) of $K$ and the indeterminates $v_i$ are of degree two (see Section \ref{pps}). The generators correspond to certain subsets of integers and the product of two generators corresponding to non-disjoint subsets is trivial, forming a symplectic basis. 

On the other hand, the cohomology ring of the real moment-angle complex is not completely understood. The group structure is known to be given by $Tor_{\Z[v_1,\ldots, v_m]}(\Z[K],\Z)$ with indeterminates $v_i$ of degree one. Theorem \ref{thethm} gives the ring structure for real moment-angle complexes over certain simplicial complexes. A  consequence of Theorem \ref{thethm} is that the multiplicative structure does not have the same nice closed form as that of moment-angle complexes. In other words, generators corresponding to non-disjoint subsets do not necessarily have trivial product.

This set of combinatorially defined generators can be identified using Bahri-Bendersky-Cohen-Gitler's Splitting Theorem \cite{MR2673742} and Welker-Ziegler-\u{Z}ivaljevi\'{c}'s wedge lemma \cite{Welker}. In this paper, we consider the case when the simplicial complex $K$ is the boundary of an $n$-gon, and describe the ring structure in terms of the combinatorial generators in Theorem \ref{thethm}. 

In full generality, for a simplicial complex $K$ on $m$ vertices, polyhedral product spaces $Z_K(\underline{X},\underline{A})$ are defined in terms of a collection of pairs of spaces $(\underline{X},\underline{A})=\{X_i,A_i\}_{i=1}^m$. The ring structure for the real moment-angle complex $Z_K(D^1,S^0)$ is particularly useful in that the cohomology ring of the more general polyhedral product $Z_K(\underline{CA},\underline{A})$ when $CA_i$ is the cone on $A_i$, can be described in terms of the ring structure of $H^*(Z_K(D^1,S^0))$ and $H^*(A)$ \cite{BBCG4}. 

Moreover, this problem of understanding the cohomology ring of a real moment-angle complex has connections to studying the topology of intersections of quadrics  and real coordinate subspace arrangements, specifically the case when $K$ is the pentagon is discussed in \cite{MR3073929}. The cohomology of real moment-angle complexes and related spaces has also been studied in \cite{CP2}, in the case of rational coefficients.

In Section \ref{sec:consequences}, we will illustrate the main theorem with some examples.   As a corollary we will see that, even though real moment-angle complexes over an $n$-gon are orientable surfaces, the combinatorial generators do not form a symplectic basis.  



\

\textbf{Acknowledgements.} This work is part of the author's doctoral dissertation at the City University of New York Graduate Center. The author would like to thank Martin Bendersky for his guidance throughout this research.

\section{Polyhedral Product Spaces}\label{pps}
In this section, we will give a brief introduction to polyhedral products, moment-angle complexes, and real moment-angle complexes, with an emphasis on the multiplicative structure of their respective cohomology rings.

Let $[m]=\{1,2,\ldots, m\}$ denote the set of integers from $1$ to $m$. An \emph{abstract simplicial complex}, $K$, on $[m]$ is a subset of the power set of $[m]$, such that:
\begin{enumerate}
\item $\emptyset \in K$.
\item If $ \sigma \in K$ with $\tau \subset \sigma$, then $\tau \in K$.
\end{enumerate}
 An $n$-simplex is the full power set of $[n+1]$ and is denoted $\Delta^{n}$. Associated to an abstract simplicial complex is its \emph{geometric realization}, denoted $\mathcal{K}$ or $|K|$ (also called a geometric simplicial complex). A (geometric) $n$-simplex, $\Delta^n$, is the convex hull of $n+1$ points. 

We do not assume $m$ is minimal, i.e. there may exist $[n] \subsetneq [m]$ such that $K$ is contained in the power set of $[n]$. 

Let $I$ be a subset of $[m]$. The \emph{full subcomplex of $K$ in  $I$} is denoted $K_I$. It is a simplicial complex on the set $I$ and defined $$K_I := \{\sigma \in K| \sigma \subset I\}.$$
It is often called the restriction of $K$ to $I$ in the literature.

 Given an abstract simplicial complex $K$, let $\mathcal{S}_K$ be the category with simplices of $K$ as the objects and inclusions as the morphisms. In particular, for $\sigma, \tau \in ob(\mathcal{S}_K)$, there is a morphism $\sigma \rightarrow \tau$ whenever $\sigma \subset \tau$. Define $\mathcal{CW}$ to be the category of CW-complexes and continuous maps. Define $(\underline{X},\underline{A})$ to be a collection of pairs of CW-complexes $\{(X_i,A_i)\}_{i=1}^m$, where $A_i$ is a subspace of $X_i$ for all $i$.

  \
  
{\begin{defn} \label{func}
Given an abstract simplicial complex $K$ on $[m]$, simplices $\sigma, \tau$ of $K$  and a collection of pairs of CW-complexes  $(\underline{X},\underline{A})$, define a diagram $D:\mathcal{S}_K \rightarrow \mathcal{CW}$ given by 
 $$D(\sigma)= \prod \limits_{i\in [m]}Y_i \qquad \text{  where } \displaystyle Y_i=\begin{cases} X_i & i \in \sigma \\ A_i & i \in [m]\backslash \sigma \end{cases}$$
For a morphism $f:\sigma \rightarrow \tau$, the functor $D$ maps $f$ to $\iota:D(\sigma)\rightarrow D(\tau)$ where $\iota$ is the canonical injection.

The \emph{polyhedral product space}  is defined as $$Z_K(\underline{X},\underline{A}):=\underset{\sigma \in K}{colim } D(\sigma)=\bigcup_{\sigma \in K} D(\sigma) $$ and is topologized as a subspace of $ \displaystyle  \prod_{i \in [m]} X_i$.\end{defn}}
  Notice that it suffices to take the colimit over the maximal simplices of $K$. In fact, simplicial complexes can be defined by their maximal simplices and this description will be used throughout. In the case where $(X_i,A_i)=(X,A)$ for all $i$, we write $Z_K(X,A)$. 
  
  Some examples of polyhedral products are moment-angle complexes $Z_K(D^2,S^1)$, which have the homotopy type of the complement of a complex
coordinate subspace arrangement, and Davis-Januszkiewicz spaces $Z_K(\mathbb{C}P^\infty, *)$, which have the Stanley-Reisner ring as cohomology ring. For a simple example, consider the following.
  \begin{ex} \label{helpful}
  Let $K$ be the boundary of a $2$-simplex with vertices labelled $1, 2, 3$. 
$$\begin{array}{ccl}Z_K(D^1,S^0)
        &=& D(\{1,2\}) \cup D(\{1,3\}) \cup  D(\{2,3\}) \\
        & = & D^1\times D^1 \times S^0 \cup D^1 \times S^0 \times D^1 \cup S^0 \times D^1 \times D^1 \\
        &=& \partial (D^1 \times D^1 \times D^1) \\
 & \cong & S^2 \end{array}$$ 
 
 In general, $Z_{\partial \Delta^{m}}(D^1,S^0) \cong S^m $ (see examples in \cite{MR2673742}).  \end{ex}
 Next we will define the polyhedral smash product, a space analogous to the polyhedral product with the smash product operation in place of the Cartesian product. Define $\mathcal{CW}_*$ to be the category of based CW-complexes and based continuous maps.
 \begin{defn}
Let the CW-pairs $(\underline{X}, \underline{A})$ be pointed. Likewise, define a functor  $\widehat{D}(\sigma):\mathcal{S}_K \rightarrow \mathcal{CW}_*$ by
$$\widehat{D}(\sigma)=\wedge Y_i \qquad \text{ where }\displaystyle Y_i=\begin{cases} X_i & i \in \sigma \\ A_i & i \notin \sigma \end{cases}$$ Then the \emph{polyhedral smash product} is  $$\widehat{Z}_K(\underline{X},\underline{A})=\bigcup \widehat{D}(\sigma)$$
\end{defn} 
For the remainder of the paper, we will assume that $(X_i,A_i)$ is a pair of pointed CW-complexes where $A_i$ is a subspace of $X_i$.

The following theorem of Bahri, Bendersky, Cohen and Gitler (BBCG) gives a stable decomposition of a polyhedral product.
\begin{thm}[Splitting Theorem, \cite{MR2673742}]\label{splitting}
Let $(\underline{X_I},\underline{A_I}) =\{(X_i,A_i)\}_{i\in I}$. 
Then  \begin{displaymath} \Sigma Z_K(\underline{X},\underline{A}) \simeq \Sigma \bigvee_{I\subset [m]} \widehat{Z} _{K_I}(\underline{X_I},\underline{A_I})) \end{displaymath}
where  $\Sigma$ denotes the reduced suspension.
\end{thm}

 In \cite{MR2673742}, the authors apply the wedge lemma from \cite{Welker} to polyhedral smash products and obtain the following:
\begin{thm}[Wedge Lemma, \cite{Welker}]\label{wedge}
 If $X_i$ is contractible for all $i$, then $$\widehat{Z}_K(\underline{X},\underline{A}) \simeq \Sigma |K| \wedge A^{\wedge [m]} \simeq |K| \ast A^{\wedge [m]}$$
 where $A^{\wedge [m]}=A_1 \wedge \ldots \wedge A_m$
\end{thm}

Since $S^0$ serves as an identity for the smash product operation, computing the cohomology groups of real moment-angle complexes becomes a combinatorial process that involves examining only the simplicial complex.   This follows from the previous two theorems.
\begin{equation}\label{realmac} H^*(Z_K(D^1,S^0))= \bigoplus\limits_{I\subset [m]} H^*(\Sigma \mathcal{K_I}) 
\end{equation}
The generators of the cohomology ring are given by the subsets of $[m]$ that yield a noncontractible full subcomplex of $K$ after suspension, which we call the \emph{combinatorial generators}.

To describe the ring structure of the cohomology of the moment-angle complex, we will  introduce some notation. The graded ring  $\mathbb{Z}[m]$ is the polynomial ring on $m$ variables $\mathbb{Z}[v_1,v_2,\ldots,v_m]$        with   $\deg(v_i)=2$.
\begin{defn}\label{SRring} The \emph{Stanley-Reisner} ring (or face ring) of the simplicial complex $K$ is the quotient of $\mathbb{Z}[m]$ by the ideal generated by square-free monomials associated to nonfaces of $K$
     $$\mathbb{Z}[K] :=  \mathbb{Z}[m] / \langle v_{i_1} v_{i_2} \ldots v_{i_k} \: | \: \{i_1,i_2,\ldots,i_k\}\notin K \rangle $$   
     \end{defn}
The following was first proved by Franz in \cite{Franz} and stated in terms of smooth toric varieties. Another proof was later given by Baskakov, Buchstaber, Panov in \cite{MR2117435}.
    \begin{thm}[Franz, \cite{Franz}]\label{BBP} The cohomology ring of the moment-angle complex $Z_K(D^2,S^1)$ is given by
    \begin{align*}
        H^*(Z_K(D^2,S^1)) \cong\text{Tor}_{\mathbb{Z}[m]}(\mathbb{Z}[K],\mathbb{Z})
    \end{align*}
    \end{thm}

A description of the multiplicative structure in terms of full subcomplexes comes from Hochster's theorem in commutative algebra on the $Tor$-module \cite{hochster}. We obtain the following analogous formula
$$H^k(Z_K(D^2,S^1)) \cong \bigoplus_{J\subset [m]} \widetilde{H}^{k-|J|-1}(K_J)$$
%
%
%
%
For the multiplicative structure, take classes $\alpha \in {H}^{i}(Z_K(D^2,S^1))$ and $\beta \in \widetilde{H}^{k}(Z_K(D^2,S^1))$. Then $\alpha$ corresponds to some class in $\widetilde{H}^{i-|J|-1}(K_J)$ for some subset  $J\subset[m]$ , and similarly $\beta$ to some class in $\widetilde{H}^{k-|L|-1}(K_L)$ for some $L\subset [m]$. Their product is
\[ \alpha \smile \beta =\begin{cases}
\gamma  & \text{ if } J \cap L=\emptyset  \\
0 &  \text{ if } J \cap L\neq \emptyset
\end{cases}
\]
for some  $\gamma$ coming from  $ \widetilde{H}^{i+k-|L|-|J|-1}(K_{J \cup L})$. See \cite{panov} for more details.

\subsection{The BBCG spectral sequence}\label{section:generalproduct}
We will use a spectral sequence developed by BBCG \cite{BBCG4}. It gives a K{\"u}nneth-like formula for the cohomology of a polyhedral product as long as the pairs $(\underline{X},\underline{A})$ satisfy the following freeness condition.

\begin{defn} \label{free}
Given the pair $(X_i,A_i)$, the associated long exact sequence is given by
$$  \ldots \overset{\delta}{\rightarrow} \widetilde{H}^*(X_i/A_i) \overset{g}{\rightarrow} H^*(X_i) \overset{f}{\rightarrow} H^*(A_i) \overset{\delta}{\rightarrow} \widetilde{H}^{\ast+1}(X_i/A_i) \overset{g}{\rightarrow} \ldots $$
The pair is said to satisfy the \emph{strong $h^*$ freeness} condition if there are free $h^*-$modules $E_i, B_i, C_i$ and $W_i$ satisfying
$$\begin{array}{ccccc}
H^*(A_{i})&=&B_{i}\oplus E_{i}\\
H^*(X_{i})&=&B_{i}\oplus C_{i}\\
\widetilde{H}^*(X_i/A_i)&=& C_i \oplus W_i\\
 \end{array}$$
where $W_i$ is $sE_i$, the suspension of $E_i$. Additionally, assume $1\in B_i$, and for $b\in B_i, c \in C_i, e \in E_i, w \in W_i=sE_i$, we have 
\begin{center} $b \overset{f}{\mapsto} b \overset{\delta}{\mapsto} 0, \qquad c \overset{g}{\mapsto} c \overset{f}{\mapsto} 0, \qquad e \overset{\delta}{\mapsto} w \overset{g}{\mapsto} 0$. \end{center} 
\end{defn}
Before defining the spectral sequence, we will give some notation and recall the definition of a half smash product:
\begin{enumerate}
\item for $\sigma =\{i_1,\ldots, i_k\}$, define $\widehat{X}^{ \sigma} := X_{i_1} \wedge \ldots \wedge X_{i_k}$  and ${A}^{\sigma}=A_{i_1} \times \ldots \times A_{i_k}$
\item the complement of a set $\sigma \subset [m] $ is  $\sigma^c = [m] \backslash \sigma$ 
\item given a basepoint $x_0 \in X$, the right half smash product $X \rtimes Y = (X \times Y) / (x_0 \times Y)$
\item for a subset $I$ and a simplex $\sigma$ such that $\sigma \subset I$, define $$Y^{I, \sigma}:= \bigotimes \limits_{i\in \sigma} C_i \otimes \bigotimes \limits_{i \in I-\sigma} B_i$$
\end{enumerate}

Choosing a lexicographical ordering for the simplices of $K$ gives a filtration of the associated polyhedral product space and polyhedral smash product, which in turn leads to a spectral sequence converging to the reduced cohomology of $Z_K(\underline{X},\underline{A})$ and a spectral sequence converging to the reduced cohomology of $\widehat{Z}_K(\underline{X},\underline{A})$. The $E_1^{s,t}$ term for $Z_K(\underline{X},\underline{A})$ has the following description.

\begin{thm}[Bahri, Bendersky, Cohen and Gitler \cite{BBCG4}]\label{bbcgss} 
 There exist spectral sequences $$E^{s,t}_r \rightarrow H^*(Z_K(\underline{X},\underline{A}))$$ $$\widehat{E}^{s,t}_r \rightarrow H^*(\widehat{Z}_K({X},{A}))$$ with $ E^{s,t}_1 = \widetilde{H}^t((\widehat{X/A})^{ \sigma} \rtimes {A}^{\sigma^c})$ and $\widehat{E}^{s,t}_1 =   \widetilde{H}^t((\widehat{X/A})^{ \sigma} \wedge \widehat{A}^{ \sigma^c})$ where  $s$ is the index of $\sigma$ in the lexicographical ordering and the differential $d_r : E_r^{s,t} \rightarrow E_r^{s+r,t+1}$ is induced by the coboundary map $\delta : E \rightarrow W=sE$. Moreover, the spectral sequence is natural for embeddings of simplicial maps with the same number of vertices and with respect to maps of pairs. The natural quotient map $$Z_K(\underline{X},\underline{A}) \rightarrow \widehat{Z}_K(\underline{X},\underline{A})$$ induces a morphism of spectral sequences and the Splitting Theorem (\ref{splitting}) induces a morphism of spectral sequences. \end{thm}

Following \cite{BBCG4}, Definition \ref{free} and the K{\"u}nneth Theorem imply that the entries $  \widetilde{H}^t((\widehat{X/A})^{ \sigma} \wedge \widehat{A}^{ \sigma^c})$ in the first page of the spectral sequence for $\widehat{Z}_K(\underline{X},\underline{A})$ decompose as a direct sum of spaces $W^N \otimes C^S \otimes B^T \otimes E^J$ such that $N \cup S =\sigma$, $T \cup J = \sigma^c$ and $N, S, J, T$ are disjoint. We have that $S$ is a simplex in $K$ as $N \cup S$ is a simplex in $K$. Since the differential is induced by the coboundary $\delta : E \rightarrow W$, consider all the possible summands $W^N \otimes C^S \otimes B^T \otimes E^J$ for $S$ and $T$ fixed. It must be the case that $N$ is a simplex in $K$ and that $N$ is a subset of $[m] \backslash (S \cup T)$. Therefore all such $N$ correspond to simplices in the link of $S$ in $K$ restricted to the vertex set $[m] \backslash (S \cup T)$.

  \begin{thm}[Bahri, Bendersky, Cohen and Gitler \cite{BBCG4}]\label{BBCGSS}
Let $(\underline{X},\underline{A})$ satisfy the decomposition described in Definition \ref{free} $$\begin{array}{ccccc}
H^*(A_{i})&=&B_{i}\oplus E_{i}\\
H^*(X_{i})&=&B_{i}\oplus C_{i}\\
 \end{array}$$
Then $$H^*(Z_K(\underline{X},\underline{A})) = \bigoplus_{I\subset [m], \sigma\subset I} E^{I^c}\otimes Y^{I, \sigma}\otimes \widetilde{H}^*(\Sigma| lk(\sigma)_{I^c}|) $$

where:

\begin{enumerate}
\item $\sigma$ is a simplex in $K$,
  \item  $lk(\sigma)_{I^c}=\{\tau \subset [m]\backslash I \;|\; \tau \cup \sigma \in K\}$ is the link of $\sigma$ in $K$ restricted to the set $[m]\backslash I$,
 \item $Y^{I, \sigma}= \bigotimes \limits_{i\in \sigma} C_i \otimes \bigotimes \limits_{i \in I-\sigma} B_i$, and
 \item $\widetilde{H}^*(\Sigma \emptyset)=1$.
\end{enumerate}
\end{thm}
\begin{thm}[Bahri, Bendersky, Cohen and Gitler \cite{BBCG4}]\label{smashSS}
Let $$\begin{array}{ccccc}
\widetilde{H}^*(A_{i})&=&\widetilde{B}_{i}\oplus E_{i}\\
\widetilde{H}^*(X_{i})&=&\widetilde{B}_{i}\oplus C_{i}\\
 \end{array}$$
Then $$H^*(\widehat{Z}_K(\underline{X},\underline{A})) = \bigoplus_{I\subset [m], \sigma\subset I} E^{I^c}\otimes Y^{I, \sigma}\otimes \widetilde{H}^*(\Sigma| lk(\sigma)_{I^c}|) $$

where:
\begin{enumerate}
\item $\sigma$ is a simplex in $K$,
 \item $lk(\sigma)_{I^c}=\{\tau \subset [m]\backslash I \;|\; \tau \cup \sigma \in K\}$ is the link of $\sigma$ in $K$ restricted to the set $[m]\backslash I$, 
 \item $Y^{I, \sigma}= \bigotimes \limits_{i\in \sigma} C_i \otimes \bigotimes \limits_{i \in I-\sigma} \widetilde{B}_i$ where $\widetilde{B}_i=B_i \backslash \{1\}$, 
 \item $\widetilde{H}^*(\Sigma \emptyset)=1$.
\end{enumerate}
\end{thm}

 A description of the ring structure in $H^*(Z_K\xa)$ is given using the decomposition from Theorems \ref{BBCGSS} and \ref{smashSS}. It is induced by a pairing involving links $$ \widetilde{H}^*(\Sigma | lk(\sigma_1)|_{I_1^c}) \otimes \widetilde{H}^*(\Sigma | lk(\sigma_2)|_{I_2^c}) \rightarrow \widetilde{H}^*(\Sigma | lk(\sigma_3)|_{I_3^c})$$ defined in terms of the $*$-product, introduced in \cite{BBCGcup}, where $I_3$ and $\sigma_3$ are defined in terms of $\sigma_1, \sigma_2, I_1$ and $I_2$.  
 \begin{thm}[Theorem 6.1 in \cite{BBCG4}] 
 Two classes \begin{align*}\alpha, \beta \in & H^*(Z_K(\underline{X},\underline{A}))\\& = \bigoplus_{I\subset [m], \sigma\subset I} E^{[m]-I}\otimes C^\sigma \otimes B^{I-\sigma}\otimes \widetilde{H}^*(\Sigma  \text{lk}(\sigma)_{I^c}),\end{align*}  
are of the form 
$$\begin{array}{ccc}
\alpha &=&a_1 \otimes a_2 \otimes \ldots \otimes a_m \otimes n_\alpha \\
\beta &=&b_1 \otimes b_2 \otimes \ldots \otimes b_m \otimes n_\beta
 \end{array}$$
 where $n_\alpha \in \widetilde{H}^*(\Sigma \text{lk}(\sigma)_{I^c})$ and $n_\beta \in \widetilde{H}^*(\Sigma \text{lk}(\tau)_{J^c})$. 

The cup product of $\alpha$ and $\beta$ is given in terms of the $*$-product and a componentwise product induced by the multiplicative structure of $H^*(X_i)$ and $H^*(A_i)$.
\end{thm}

 For the pair of spaces $(CA_i,A_i)$, where $CA_i$ is the cone on $A_i$, the modules are given by $B_i=1$, $C_i=0$ and $E_i=\widetilde{H}^*(A)$. The links are all of the form $ K_I$ for $I\subset [m]$. Therefore, it can be seen from Theorem \ref{BBCGSS} that the product structure in  $H^*(Z_K(\underline{CA},\underline{A}))$ can be described in terms of the product structure in $H^*(A)$ and $\widetilde{H}^*(\Sigma \mathcal{K}_I)$. 
 
 Due to the decomposition in Equation \ref{realmac} and work in \cite{BBCGcup},  the ring structure in  $H^*(Z_K(\underline{CA},\underline{A}))$ can be described in terms of the ring structure in $H^*(A)$ and $H^*(Z_K(D^1,S^0))$. 
 
 \begin{thm}[Theorem 1.9 in \cite{BBCGcup}]\label{generalringPP}
 Assume that any finite product
of $A_i$ with $Z_{K_I}(D^1,S^0)$ for all $I$ satisfies the strong form of the K{\"u}nneth Theorem.  Then
the cup product structure for the cohomology algebra $H^*
(Z_K(\underline{CA},\underline{A}))$ is a functor of the
cohomology algebras of $A_i$, and $Z_{K_I}(D^1,S^0)$ for all $I$.
 \end{thm}

\section{Multiplicative structure of $H^*(Z_K(D^1,S^0))$}

Recall from Equation \ref{realmac} that each subset $I$ of $[m]$ such that the full subcomplex $\mathcal{K}_I$ is not contractible  corresponds to a generator of $H^*(Z_K(D^1,S^0))$.

To compute the cohomology of a real moment-angle complex, we will use a filtered chain complex induced by the long exact sequence of the pair $(D^1,S^0)$, denoted $C_K$ and constructed in \cite{BBCG4}. For $(X_i,A_i) = (D^1,S^0)$, let $\widetilde{H}^*(A_k)=\widetilde{H}^*(S^0)$ be generated by $t_k$ and $\widetilde{H}^*(X_k/A_k)=\widetilde{H}^*(S^1)$ be generated by $s_k$.
\begin{defn}\label{chain}
The chain complex $C(K_I)$ is generated by $y_\sigma := \otimes y_i$ where $\sigma \in K_I$ and $$y_i= \begin{cases}
  s_{i} &  i\in \sigma \\
  t_i &   i \in I-\sigma \\
  1 &  k \notin I
  \end{cases}$$
  The differential is defined by $$d_I(y_\sigma)=\sum_\tau (-1)^{n(\tau)} y_\tau$$ where $\sigma \subset \tau \in K_I$ and $\tau=\sigma \cup v$ for some vertex $v \in I$. The integer $n(\tau)$ is defined by the usual sign convention of a graded derivation. In particular, the coboundary $\delta$ acts on each factor of $y_\sigma$ by $\delta(s_i)=0$ and $\delta(t_i)=s_i$, and every time it passes an $s_i$ a factor of $(-1)$ is introduced.
  \end{defn}
  Then $$C_K=\bigoplus_{I\subset [n]} C(K_I)$$ and $H^*(C_K)=H^*(Z_K(D^1,S^0))$.

   It follows from work of Li Cai in \cite{cai2017} that the chain level cup product of two generators is induced by the following
  $$s_i \smile s_i=0 , \hspace{1cm} t_i \smile t_i=t_i, \hspace{1cm} s_i \smile t_i=s_i, \hspace{1cm} t_i \smile s_i=0.$$

  \subsection{Boundary of a polygon}
  We will consider the case of $K$ the boundary of a polygon. By Theorem \ref{realmac}, we need to consider all subsets of $[n]$ to find the cohomology groups. By convention, when $I$ is the empty set, $H^*(\Sigma \mathcal{K}_I)=1$. The suspension of the whole complex $K$ is a degree two generator. The following lemma gives the generators of degree one.

  \begin{lem} \label{first} Suppose $K$ is the boundary of an $n$-gon. Let $I=I_1 \sqcup I_2 \sqcup \ldots \sqcup I_p$ be a subset of $[n]$ such that $\mathcal{K}_I$ has exactly $p$ maximal connected components,  $\mathcal{K}_I\simeq \bigvee_{p-1} S^0$. Then
	$$H^1(\Sigma \mathcal{K}_I)=\faktor{\langle \sum_{i\in I_1}y_{\{i\}}, \sum_{i\in I_2}y_{\{i\}}, \ldots, \sum_{i\in I_{p}}y_{\{i\}}\rangle}{\langle \sum_{i\in I}y_{\{i\}}\rangle}$$
	\end{lem}
	\begin{proof}
	Let $I_1=\{i_1, \ldots, i_c\}$. If $i_1=i_c$, then $d(y_{\{i_1\}})=0$ and $y_{\{i_1\}}$ is clearly a cocycle. If $i_1\neq i_c$, then the differential will not be trivial. If $+y_{\{e\}}$ for some edge $e\in K_{I_1}$ appears as a summand in the image of $d(y_{\{v\}})$ for some vertex $v\in I_1$, then $e=\{v-1,v\}$ or $e=\{1,n\}$ (since $y_{\{e\}}$ was positive, we could not have passed an $s$). Additionally, since $y_{\{e\}}$ was in the image of $y_{\{v\}}$, it must be the case that $v-1\in I_1$, so $-y_{\{e\}}$ is a term in the image of $y_{\{v-1\}}$ under $d$. If it had been the case that $e=\{1,n\}$, then $\{1\} \in K_I$ and $-y_{\{e\}}$ would be in the image of $y_{\{1\}}$. Since it is only possible for $y_{\{e\}}$ to be in the image of $y_{\{v-1\}}$ or $y_{\{v\}}$,  the terms $y_{\{e\}}$ cancel. This means that  $d(y_{\{i_1\}}+\ldots + y_{\{i_c\}})=0$ since $i_1, \ldots, i_c$ are all the vertices in a connected component of $\mathcal{K}_I$. Without loss of generality, the same is true for the other connected components. Lastly, $d(\emptyset)$ is the sum of $y_{\{i\}}$ for $i\in I$. 
	\end{proof}
\begin{cor}\label{cor:rank}
  If $I=I_1 \sqcup I_2 \sqcup \ldots \sqcup I_p$ is a subset of $[n]$ such that $\mathcal{K}_I$ has exactly $p$ maximal connected components with $p>1$, then  $H^1(\Sigma |\mathcal{K}_I|)$ has rank $p-1$ and a basis of generators can be chosen by picking any $p-1$ of the $p$ disjoint subsets $I_1, I_2, \ldots, I_p$.
\end{cor}	
		Next, the following lemma will show how generators coming from different subsets of $[n]$ multiply. Consider subsets $I,J \subset [n]$ such that $I\cup J=[n]$. We will employ a slight change in notation: replacing $y$'s associated to $I$ with $a$'s and $y$'s associated to $J$ with $b$'s to differentiate between generators in $C(K_I)$ and generators in $C(K_J)$. Recall that $a_{\{i\}}$ is the generator associated to the vertex $i$, whereas $a_i$ is the $i$th factor of a generator.
	
	\begin{lem}\label{lem:product}
  Suppose $a_{\{i\}} \in C(K_I)$ and $b_{\{j\}} \in C(K_J)$. Then $a_{\{i\}}=a_1\otimes a_2 \otimes \ldots \otimes \ a_n$ where $$a_k=\begin{cases}
  s_{i} &  k=i \\
  t_k &   k\in I\backslash \{i\} \\
  1 &  k \notin I
  \end{cases}$$
Define $b_{\{j\}}$ similarly. Then $$a_{\{i\}} \smile b_{\{j\}}= \begin{cases}
0 & j \in I \text{ or } i-j \neq \pm 1 \pmod{n}\\
y_{\{i,j\}} & j>i\\
-y_{\{j,i\}} & j<i \end{cases} $$
  \end{lem}
  \begin{proof}
	If $|i-j|\neq1$, then $\{i,j\}$ is not a simplex in $K$ and $a_{\{i\}} \smile b_{\{j\}}=0$. Therefore, we will now consider cases where $|i-j|=1$.
	
	Recall that  $s\smile t=s\smile 1 = s$ and $t\smile s=0$.
	
  Suppose $j \in I$. Since $j\in I$, $a_j=t_j$. In particular, in the $j$th coordinate of $a_{\{i\}} \smile b_{\{j\}}$, we will have $ a_{j} \smile b_{j}= t_{j} \smile s_{j} =0$  so $a_{\{i\}} \smile b_{\{j\}}=0$.

  Next suppose $j \notin I$.  Then $a_j=1$ and $a_j \smile b_j =s_j$. If $j=i+1$, then
  
  \noindent
  $a_{\{i\}} \smile b_{\{j\}}$

\noindent  
  $\begin{array}{cccccccc}
 =& (a_{1}\smile b_{1}) &\otimes \ldots \otimes & (a_{i} \smile b_{i}) & \otimes & (a_{j} \smile b_{j}) &\otimes \ldots \otimes & (a_{n}\smile b_{n}) \\
=&t_1 &\otimes \ldots \otimes & s_{i} \smile b_i & \otimes & 1 \smile s_{j} &\otimes \ldots \otimes & t_n  \\
=& t_1 & \otimes \ldots \otimes & s_{i} & \otimes & s_{j} & \otimes \ldots \otimes & t_n \\
=& y_{\{i,j\}}
  \end{array}$
  
  \noindent
  since the only coordinate of $b_{\{j\}}$ that is an $s$ is $b_j$ and all other coordinates are $t$ or $1$.
  
    If  $j=i-1$, since $a_i \smile b_{i-1} = (-1)^{|b_{i-i}||a_i|}( b_{i-1} \smile a_i)$, we have
    
    \noindent  
    $a_{\{i\}} \smile b_{\{j\}}$

\noindent  
  $\begin{array}{cccccccc}
=& (a_{1}\smile b_{1}) &\otimes \ldots \otimes &(-1)^{|b_j||a_i|} (a_{j} \smile b_{j}) & \otimes & (a_{i} \smile b_{i}) &\otimes \ldots \otimes & (a_{n}\smile b_{n}) \\
=&t_1 &\otimes \ldots \otimes & (-1)(1\smile s_{j} )& \otimes &   s_{i} \smile b_i &\otimes \ldots \otimes & t_n  \\
=& t_1 & \otimes \ldots \otimes & (-1) s_{j} & \otimes & s_{i}& \otimes \ldots \otimes & t_n \\
=& -y_{\{ j,i \} }
  \end{array}$
  \end{proof}
	The following theorem uses the previous lemmas to show what the only non-trivial products in $H^*(Z_K(D^1,S^0))$ are.
	 \begin{thm}\label{thethm} Let $K$ be the boundary of an $n$-gon.  Let $I=I_1 \sqcup I_2 \sqcup \ldots \sqcup I_p$ and $J=J_1 \sqcup J_2 \sqcup \ldots \sqcup J_q$ be nonempty subsets of $[n]$ such that $I\cup J =[n]$ and ${K}_I\simeq \bigvee_1^{p-1} S^0$ and  ${K}_J \simeq \bigvee_1^{q-1} S^0 $. Given generators $\alpha$ and $\beta$ of $H^*(Z_K(D^1,S^0))$ such that one is associated to some $I_g$ for $1\leq g \leq p$ and the other is associated to some $J_h$ for some  $1 \leq h \leq q$. If $\gamma$ is the second degree generator of $H^*(Z_K(D^1,S^0))$, then $\alpha \smile \beta=\pm \gamma$ if and only if the following conditions are met
\begin{itemize}
\item $I_g \nsubseteq J_h$ 
\item $J_h \nsubseteq  I_g$
\item ${K}_{I_g \cup J_h}$ is contractible
\end{itemize}
  \end{thm}

\begin{proof}
We will compute $H^*({Z}_K(D^1,S^0))$ using the chain complex described previously. 
  $$ \begin{array}{ccc}
    d(y_\emptyset) &=& y_{\{1\}} + \ldots + y_{\{n\}}\\
    d(y_{\{1\}})&=& -y_{\{1,2\}}-y_{\{1,n\}}                \\
    d(y_{\{2\}})&=& y_{\{1,2\}}-y_{\{2,3\}} \\
    d(y_{\{3\}})&=&y_{\{2,3\}}-y_{\{3,4\}} \\
    \hdots & & \\
    d(y_{\{n-1\}})&=&y_{\{n-2,n-1\}} - y_{\{n-1,n\}} \\
    d(y_{\{n\}})&=& y_{\{n-1,n\}}+y_{\{1,n\}}
  \end{array}$$
 Therefore, all the classes in $H^*({Z}(K;(D^1,S^0)))$  represented by an edge are cohomologous, except $y_{\{1,n\}}$, which is the negative. Note that if $1$ and $n$ are in $I$, then $1$ and $n$ are in $I_i$ for some $1 \leq i \leq p$ (since $\mathcal{K}_I\simeq \bigvee_1^{p-1} S^0$, it cannot be that $1$ and $n$ are in different subsets of $I$). By Corollary \ref{cor:rank}, we only need to consider all but one of the disjoint subset of $I$ and all but one of the disjoint subsets of $J$. Therefore, it suffices to only consider when $1, n \notin I_g \cup J_h$. As a consequence, the class cannot $y_{\{1,n\}}$ occur in the product $\alpha \smile \beta$.

Let $I_g=\{i_1, \ldots, i_c \}$ and $J_h=\{j_1, \ldots, j_d \}$. By Lemma \ref{first}, we have that $$\alpha = \sum_{i\in I_g} a_{\{i\}}$$

Since $I\cup J = [n]$, we have that $i_1 \neq j_1$ and $i_c \neq j_d$. 

First, suppose $I_g \cap J_h=\emptyset$. In the case that $i_c < j_1$, we must have $i_c=j_1-1$ so that there is at least one edge after expanding the product. Then there is only one nonzero term $$\alpha \smile \beta =\sum_{i\in I_g,j\in J_h} a_{\{i\}} \smile b_{\{j\}}=y_{\{i_c,j_1\}}$$ by Lemma \ref{lem:product}.
Similarly, if $j_d = i_1-1$, then $$\alpha \smile \beta =\sum_{i\in I_g,j\in J_h} a_{\{i\}} \smile b_{\{j\}}=a_{\{i_1\}}\smile b_{\{j_d\}}=-y_{\{j_d,i_1\}}$$

Secondly, suppose $I_g \cap J_h \neq \emptyset$ and that neither set is contained in the other. If $j_1 \leq i_c$, then there exists $j\in J_h $ such that $j=i_c$. Note that $j+1\in J_h$ (because $i_c \neq j_d$ and so $j\neq j_d$). Since $\{i_c, j+1\}$ is an edge and $j+1=i_c+1\notin I$, by Lemma \ref{lem:product} we have only one nonzero term $$\alpha \smile \beta = a_{\{i_c\}} \smile b_{\{j+1\}}= y_{\{i_c, i_c+1\}}$$
Similarly, if $i_1 \leq j_d$ and $j=i_1$ for some $j \in J_h$, then $j-1\notin I$ and $i_1-(j-1)=1$. Then $$ \alpha \smile \beta =a_{\{i_1\}}\smile b_{\{j-1\}}= -y_{\{i_1-1,i_1\}}$$

If $J_h \subset I_g$, then $\alpha \smile \beta=0$ by Lemma \ref{lem:product}. If $I_g \subset J_h$, then there are only two possible nonzero products between the summands of $\alpha$ and $\beta$. There exists $j\in J_h$ such that $j=i_1$. Then
$$\begin{array}{ccc} \alpha \smile \beta &=& a_{\{i_1\}}b_{\{j-1\}}+a_{\{i_c\}}b_{\{j+c\}}\\
&=&-y_{\{i_1-1,i_1\}} + y_{\{i_c,i_c+1\}} \\
&=& -\gamma+\gamma \\
&=& 0
\end{array}$$
\end{proof}

\subsection{Example and related consequences}\label{sec:consequences}

To illustrate an application of the theorem and some important consequences, we will consider the case when $K$ is the boundary of the pentagon, denoted $K_5$. Let the $1$-simplices of $K_5$ be labeled $ \{1,2\},\{2,3\},\{3,4\},\{4,5\},\{1,5\}$.

It follows from \cite{MR1104531} that the real moment-angle complex over the boundary of an $n$-gon is a closed orientable surface of genus $1+(n-4)2^{n-3}$, which means that in this example the associated real moment-angle complex has genus five. For the combinatorial generators, there are ten subsets $I$ of $[5]$ that yield a full subcomplex $K_I$ equivalent to a wedge of $0$-spheres.  The cohomology of $Z_{K_I}(D^1,S^0)$ has an identity, ten degree one generators $x_0,\ldots,x_4,w_0,\ldots,w_4$, and a degree two generator $z$, subject to a graded commutative product. The identity corresponds to the empty set. The generators $x_i$ correspond to the subsets that yield a full subcomplex of $K$ of an edge and the opposite vertex, such as the subset $I=\{1,2,4\}$. The generators $w_i$ correspond to the subsets that produce a full subcomplex of two disjoint vertices. Lastly, $z$ corresponds to the full vertex set $[5]=\{1,2,3,4,5\}$. $$H^*(Z_{K_5}(D^1,S^0))=\langle 1,w_0,\ldots,w_4,x_0,\ldots,x_4,z \: | \: x_i x_j= z \delta_{j,i+1}, x_i w_j=z \delta_{i,j} \rangle$$ where $\delta$ is the Dirac delta function and the subscripts $i, j$ are integers modulo $4$.

Notice that for subsets $J=\{2,4,5\}$ and $I=\{1,3,4\},$  $ \alpha_I \smile \alpha_J=\gamma$. This is an example where generators coming from non-disjoint subsets have a nontrivial product, unlike the ring for moment-angle complexes. The cohomology ring of $Z_K(D^1,S^0)$ is not isomorphic to the Tor-module as rings. Moreover, this application of Theorem \ref{thethm} also shows that the basis of combinatorial generators is not symplectic.

 Lastly, recall that the multiplicative structure of the cohomology of real moment-angle complexes plays an important role in  the product structure for more general polyhedral product spaces \cite{BBCGcup}. Theorem \ref{generalringPP} gives the algebra $H^*
(Z_K(\underline{CA},\underline{A}))$ in terms of the cohomology algebras of $A_i$ and $H^*
(Z_{K_I}(D^1,S^0))$.

\bibliographystyle{amsplain}
\bibliography{mybibliography}
\end{document}